\documentclass[preprint,12pt]{elsarticle}
\usepackage[utf8]{inputenc}

\usepackage{fullpage}

\usepackage{amsmath}
\usepackage{amssymb}
\usepackage{latexsym}
\usepackage{amsthm}
\usepackage[retainorgcmds]{IEEEtrantools}

\usepackage{tikz}
\usetikzlibrary{arrows,calc}

\tikzstyle{arc}=[->,shorten <=3pt, shorten >=3pt,
                 >=stealth, line width=1.1pt]
\tikzstyle{edge}=[shorten <=2pt, shorten >=2pt,
                  >=stealth]
\tikzstyle{vertex}=[circle, fill=white, draw,
                    minimum size=5pt,
                    inner sep=0pt, outer sep=0pt, fill = black]
\DeclareMathOperator{\girth}{girth}

\newtheorem{theorem}{Theorem}
\newtheorem*{theorem*}{Theorem}
\newtheorem{lemma}[theorem]{Lemma}
\newtheorem{observation}[theorem]{Observation}
\newtheorem{proposition}[theorem]{Proposition}
\newtheorem{corollary}[theorem]{Corollary}

\newtheorem*{problem}{Problem}
\newtheorem*{question}{Question}
\newtheorem{conjecture}[theorem]{Conjecture}

\usepackage{minitoc}

\usepackage{ifpdf}

\begin{document}

\begin{frontmatter}

\title{Oriented cobicircular matroids are $GSP$\tnoteref{t1}}

\tnotetext[t1]{This work is an extension of our previous pre-print: arXiv:2203.12549}

\author[FC,FernUni]{Santiago Guzm\'an-Pro}
\ead{sanguzpro@ciencias.unam.mx}

\author[FernUni]{Winfried Hochst\"attler}
\ead{winfried.hochstaettler@fernuni-hage.de}

\address[FC]{Facultad de Ciencias\\
Universidad Nacional Aut\'onoma de M\'exico\\
Av. Universidad 3000, Circuito Exterior S/N\\
C.P. 04510, Ciudad Universitaria, CDMX, M\'exico}

\address[FernUni]{FernUniversit\"at in Hagen\\
Fakult\"at f\"ur Mathematik und Informatik\\
 58084 Hagen}


\begin{abstract}
Colouring and flows are well-known dual notions in Graph Theory.
In turn, the definition of flows in graphs naturally extends to flows in
oriented matroids.
So, the colour-flow duality gives a generalization of Hadwiger's conjecture
about graph colourings, to a conjecture about coflows of oriented matroids.
The first non-trivial case
of Hadwiger's conjecture for oriented matroids reads as follows. If $\mathcal{O}$
is an $M(K_4)$-minor free
oriented matroid, then 
$\mathcal{O}$ has a now-where $3$-coflow, i.e., it is $3$-colourable in the sense of
Hochst\"attler-Ne\v{s}et\v{r}il.
The class of generalized series parallel ($GSP$) oriented matroids is
a class of $3$-colourable oriented matroids with no $M(K_4)$-minor.
So far, the only technique towards proving that all orientations of a 
class $\mathcal{C}$ of $M(K_4)$-minor free matroids are $GSP$ (and thus
$3$-colourable), has been to show that every matroid in $\mathcal{C}$
has a positive coline. Towards proving Hadwiger's conjecture for the class
of gammoids, Goddyn, Hochst\"attler, and Neudauer conjectured that
every gammoid has a positive coline. In this work we disprove this
conjecture by exhibiting an infinite class of strict gammoids that do not
have positive colines. We conclude by proposing a simpler technique
for showing that certain oriented matroids are $GSP$. In particular,
we recover that oriented lattice path matroids are $GSP$, and we show
that oriented cobicircular matroids are $GSP$.
\end{abstract}


\end{frontmatter}

\section{Introduction}

Hadwiger's Conjecture is a well-known and long open conjecture regarding proper graph
colourings. It states that for every positive integer $k$ if a graph $G$ contains no
$K_{k+1}$-minor, then $G$ is $k$-colourable. This conjecture has been proven 
true for $k\le 5$ \cite{robsertsonC13}, and remains open for larger integers. 

The notion of proper graph colourings is the dual concept of
nowhere-zero flows (NZ flows).  The latter, has a natural generalization
to oriented matroids, 
which Hochst\"attler and Ne\v{s}et\v{r}il~\cite{hochstattlerEJC27} use
to  propose a definition of the chromatic number of an oriented matroid.
It turns out that Hadwiger's conjecture can be generalized to this context;
but in this scenario, the first non-trivial case remains open and it reads as follows. 

\begin{conjecture}\label{conj:hadw}
Every (loopless) $M(K_4)$-minor free oriented matroid has a nowhere-zero $3$-coflow.
\end{conjecture}

Goddyn and Hochst\"attler \cite{goddynSM} observed that Hadwiger's
conjecture for regular oriented matroids, includes the cases $k = 4$ and
$k = 5$ of  Tutte's $k$-flow conjecture \cite{tutteCJM6, tutteCMA, tutteJCT},
which remain open.

Towards proving Conjecture~\ref{conj:hadw}, Goddyn, Hochst\"attler and Neudauer,
introduce the class of \textit{Generalized Series Parallel} ($GSP$) oriented matroids
and show that every $GSP$ oriented matroid has a NZ
$3$-coflow~\cite{goddynDM339}. 
It might be too much to hope for, but if every $M(K_4)$-free oriented matroid
is $GSP$, then Conjecture~\ref{conj:hadw} follows directly. In any case,
this raises the fundamental problem of determining when a class $\mathcal{C}$
of oriented matroids is a subclass of $GSP$ oriented matroids. To this end
and in the same work, the previously mentioned authors show that if
$\mathcal{C}'$ is a class of orientable matroids closed under minors such that
every member of $\mathcal{C}'$ has a  positive coline, then the class $\mathcal{C}$
of all orientations of matroids in $\mathcal{C}'$ is a class of $GSP$ oriented
matroids. Finally, they show that every bicircular matroid has a positive coline,
so, if $\mathcal{O}$ is an oriented bicircular matroid,  then
$\mathcal{O}$ is $GSP$ and thus it has a NZ $3$-coflow. 

Bicircular matroids are transversal matroids, and the smallest class
closed under minors that contains transversal matroids is the class of gammoids. 
In turn, the class gammoids is a class of $M(K_4)$-free orientable matroids, so 
Goddyn, Hochst\"attler and Neudauer pose the following conjecture. 

\begin{conjecture}\label{conj:coline}\cite{goddynDM339}
Every simple gammoid of rank at least two has a positive coline.
\end{conjecture}

In this work, we disprove this conjecture by exhibiting a large class of
cobicircular matroids that do not have positive colines; but we show
that nonetheless, every orientation of a cobicircular matroid is $GSP$. 

The rest of this work is organized as follows. In Section~\ref{sec:prelim},
we introduce all concepts needed to state  Conjecture~\ref{conj:coline}.
In Section~\ref{sec:counter}, we introduce bicircular
matroids and prove the necessary results to exhibit a class of counterexamples to
the previously mentioned conjecture. In Section~\ref{sec:GSP},
we prove that all orientations of cobicircular matroids are $GSP$. 
Finally, in Section~\ref{sec:conclusions} we conclude this work
by posing some further questions and problems that arose from 
this work.

\section{Preliminaries}
\label{sec:prelim}

We assume basic familiarity with matroids and with oriented matroids, 
standard references are~\cite{bjorner1993} and \cite{oxley1992}.

\subsection{$GSP$ oriented matroids}

Consider an oriented matroid $O$ with ground set $E$ and collection
of signed circuits $\mathcal{C}$. The \textit{signed vector} of a signed circuit
$C = (C^+, C^-)$ is the characteristic vector of $C$ in $\{0,1,-1\}^{E}$. In 
other words, $C(e) = 1$ if $e\in C^+$; $C(e) = -1$ if $e\in C^-$; and
$C(e) = 0$ if $e\not\in C$. 
The \textit{flow lattice} of $O$, denote by $\mathcal{F}_O$
is the integer lattice generated by the signed vectors of the circuits of $O$.
In symbols, 
\[
\mathcal{F}_O = \Biggl\{\sum_{C\in\mathcal{C}} \lambda_CC|~\lambda_C\in
\mathbb{Z}\Biggl\}.
\]
We call any element $x\in \mathcal{F}_O$ a flow of $O$. The flow lattice
of uniform matroids is characterized in \cite{hochstattlerCDM2}. In particular,
the following statements hold. 

\begin{lemma}\label{lem:uniformflows}
Let $O$ be an orientation of a rank $r$ uniform matroid of at least $r+2$
elements. For any pair of elements $e,f\in E$, the following statements hold:
\begin{enumerate}
	\item if $r$ is even, then there is a flow $x\in \mathcal{F}_O$ such
	that $x(e) = 1$ and $x(e') = 0$ for every $e'\neq e$, and
	\item if $r$ is odd, then there is a flow $x\in \mathcal{F}_O$ such
	that $|x(e)| = |x(f)| = 1$ and $x(e') = 0$ for every $e'\not\in\{e,f\}$.
\end{enumerate}
\end{lemma}
\begin{proof}
The descriptions of flow lattice used in this proof are taken from 
\cite{hochstattlerCDM2}. With out loss of generality suppose that $E = [n]$.
The first statement holds because
the flow lattice of a uniform matroid of even rank $r$, $r\le n-2$, is
$\mathbb{Z}^{n}$. Regarding the flow lattice of oriented uniform matroids of
odd rank, there are two possibilities. The first option, is that there is a vector
$v\in\{1,-1\}^n$ such that $\mathcal{F}_O$ is the orthogonal 
space $\{v\}^\perp$ intersected with $\mathbb{Z}^n$%
\footnote{This case occurs if and only if $\mathcal{O}$ is a reorientation of a 
so-called neighbourly oriented matroid}.  In this case, 
for any $i,j\in[n]$ with $i\neq j$, consider the vector $x$ defined
as follows: $x(k) = 0$ for any $k\not\in \{i,j\}$; $x(i) = v(i)$ and $x(j) = -v(j)$. 
Clearly, $x\in\{v\}^\perp\cap \mathbb{Z}^n$, and so, this case is settled. 

The other option, is that the flow lattice of an oriented uniform matroid
of odd rank is characterized as follows. Let $v_1$ be the all $1$ vector
in $\{0,1,-1\}^n$. The flow lattice  $\mathcal{F}_O$ is the set of points
$x\in\{0,1,-1\}^n$ such that the product $x^Tv_1$ is even.
In this case, for any $i,j\in [n]$
with $i\neq j$, consider the vector $x$ where $x(i) = x(j) = 1$, and $x(k) = 0$
for every $k\not\in\{i,j\}$. Clearly, $x\in \mathcal{F}_O$ and thus, the claim
holds in both cases. 
\end{proof}

Analogously, the \textit{coflow lattice} of an oriented matroid $O$ with
signed cocircuits $\mathcal{D}$, is the integer lattice generated by
the signed vectors of the cocircuits of $O$.  Clearly, the coflow lattice
of $O$ is the flow lattice of the dual $\mathcal{O}^*$, and so, we denote
by $\mathcal{F}_{O^*}$ the coflow lattice of $O$. An element
$x\in \mathcal{F}_{O^*}$ is a coflow of $O$. Such an element is 
a \textit{nowhere-zero-$k$ coflow} if $0<|x(e)| < k$ for every $e\in E$.

Having introduced all this nomenclature, we restate the first non-trivial
open case of Hadwiger's conjecture for oriented matroids (and its
dual statement in terms of flows, that follows because $M(K_4)$
is self-dual).

\begingroup
\def\thetheorem{\ref{conj:hadw}}
\begin{conjecture}
Every (loopless)  $M(K_4)$-minor free oriented matroid has a nowhere-zero $3$-coflow. 
Equivalently, 
every (coloop free) $M(K_4)$-minor free oriented matroid has a nowhere-zero $3$-flow.
\end{conjecture}
\addtocounter{theorem}{-1}
\endgroup

This conjecture restricted to graphic matroids, is equivalent to the statement:
every graph with no $K_4$ minor is $3$-colourable \cite{goddynDM339}.
This equivalent statement can be easily proven, even in the case of regular
oriented matroids, since the class of $M(K_4)$-minor
free regular oriented matroids correspond to series parallel graphs. 
Thus, with a simple inductive argument one can show that every graph with no
$K_4$-minor is $3$-colourable. 

In an attempt to mimic the previous technique for oriented matroid and 
$NZ$-$3$ coflows, Goddyn, Hochst\"attler, and Neudauer introduce the
class of \textit{generalized series parallel} ($GSP$) oriented matroids
\cite{goddynDM339}. An oriented matroid $O$ is $GSP$ if
every simple minor of $O$ has a $\{0,1,-1\}$-coflow with at most
two non-zero entries.  In the same work, the authors show that every $GSP$ oriented
matroid has a NZ $3$-coflow. 

\begin{proposition}\cite{goddynDM339}
Let $O$ be a $GSP$ oriented matroid. If $O$ has no loops, then 
$O$ has a NZ-$3$-coflow.
\end{proposition}

Not much is known about $GSP$ oriented matroids. In particular, 
it is not known whether $GSP$ oriented matroids are closed under
duality. For this reason,  we introduce the dual class of generalized series
parallel oriented matroids. An oriented matroid $O$
is \textit{$coGSP$} if every cosimple minor of $O$ has 
 a $\{0,1,-1\}$-flow with at most two non-zero entries.

\subsection{Positive colines}

Let $M$ be a matroid. A \textit{copoint} of $M$ is a hyperplane, that is, 
a flat of codimension $1$. A \textit{coline} of $M$ is a flat of codimension $2$. 
If a coline $L$ is contained in a copoint $H$, we say that $H$ is a \textit{copoint on}
$L$. It is not hard to notice that if $L$ is a coline of a matroid $M$ then there
is a partition $(H_1,\dots, H_k)$ of $E(M)\setminus  L$ such that the copoints
on $L$ are $L\cup H_i$ for $i\in\{1,\dots k\}$. Furthermore, this partition is unique
(up to permutations) so we call it the \textit{copoint partition} of $L$, and
define the \textit{degree} of $L$ to be $k$. A class $H_i$ is \textit{singular}
if $|H_i| = 1$; otherwise it is a \textit{multiple} class.
A coline $L$ is \textit{positive} if there are more
singular that multiple classes  in its copoint partition.
It turns out that positive colines and $GSP$ oriented matroids are related as follows.

\begingroup
\def\thetheorem{\cite{goddynDM339}}
\begin{proposition}
Let $\mathcal{C}$ be minor closed class of orientable matroids. 
If every simple matroid in $\mathcal{C}$ has a positive coline, then
every orientation of a matroid in $\mathcal{C}$ is $GSP$.
\end{proposition}
\addtocounter{theorem}{-1}
\endgroup

Using this proposition, Goddyn, Hochst\"attler, and Neudauer
\cite{goddynDM339} show that every orientation of a bicircular matroid
is $GSP$. Every bicircular matroid is a transversal matroid, and the class
of gammoids is the smallest dually and minor closed class that contains transversal 
matroids \cite{ingletonJCTB15}. These facts motivate Goddyn, Hochst\"attler,
and Neudauer to conjecture that every orientation of a gammoid is $GSP$. 
Furthermore, they conjecture that the following statement is true.

\begingroup
\def\thetheorem{\ref{conj:coline}}
\begin{conjecture}\cite{goddynDM339}
Every simple gammoid of rank at least two  has a positive coline.
\end{conjecture}
\addtocounter{theorem}{-1}
\endgroup

\subsection{Double circuits}

Circuits are the dual complements of hyperplanes. That is, if $H$ is a
hyperplane of a matroid $M$ then $E(M)\setminus  H$ is a circuit of $M^\ast$. 
A \textit{double circuit} of a matroid $M$ is a set $D$ such that $r(D) = |D|-2$
and for every element $d\in D$ the rank of $D-d$ does not decrease, i.e.\
$r(D-d) = |D| -2 = r(D)$. Dress and Lov\'asz \cite{dressC7} show that if $D$ is a
double circuit then $D$ has a partition $(D_1,\dots, D_k)$ such that the circuits
of $D$ are $D\setminus  D_i$ for $i\in\{1,\dots, k\}$. We call this partition the \textit{circuit
partition} of $D$ and say that the \textit{degree} of $D$ is $k$. 

\begin{observation}\label{obs:dcseries}
Let $M$ be a matroid and  $D\subseteq E(M)$ a double circuit of $M$.
If $D$ is a degree $k$ double circuit, then $M[D]$ is a series 
extension of $U_{k-2,k}$.
\end{observation}
\begin{proof}
With out loss of generality suppose that $D = E(M)$, and let
$(D_1,\dots, D_k)$ be the circuit partition of $D$.  If $|D_i| = 1$ for every
$i\in\{1,\dots, k\}$, then $M \cong U_{k-2,k}$. Now, the claim follows 
by a straightforward induction over the difference  $|D| - k$.
\end{proof}

Similar to how circuits are the dual complements of copoints, 
double circuits are the complements of colines. Moreover, the 
copoint partitions and circuit partitions relate as follows.

\begin{observation}\label{obs:coline-doublec}
Let $M$ be a matroid, $L\subseteq E(M)$ and $(H_1,\dots H_k)$ a partition of
$E(M)\setminus  L$. Then, $L$ is  a coline of $M$ with copoint partition
$(H_1,\dots, H_k)$ if and only if $E(M)\setminus  L$ is a double circuit of $M^\ast$
with circuit partition $(H_1,\dots, H_k)$.
\end{observation}

A \textit{positive double circuit} $D$ is a double circuit with more singular
than multiple classes in its circuit partition. By Observation~\ref{obs:coline-doublec},
a matroid $M$ has a positive double circuit if and only if $M^\ast$
has a positive coline. Since gammoids are closed under duality, 
the following conjecture is equivalent to Conjecture~\ref{conj:coline}.

\begin{conjecture}\cite{goddynDM339}\label{con:posdouble}
Every cosimple gammoid of corank at least two  has a positive double circuit.
\end{conjecture}

\section{Bicircular matroids and double circuits}
\label{sec:counter}

Every bicircular matroid is a transversal matroid~\cite{matthewsQJMOS28},
and so, every bicircular matroid is a gammoid. In this section,  we disprove 
Conjecture~\ref{conj:coline} by showing that its dual statement,
Conjecture~\ref{con:posdouble},
does not hold for bicircular matroids. To do so, we begin by briefly introducing
the class of bicircular matroids.

A standard reference for graph theory is \cite{bondy2008}. In particular, 
given a graph $G$ and a subset of edges $I$  we denote by
$G[I]$ the subgraph of $G$ induced by $I$. That is,  $G[I]$ is the subgraph
of $G$ with edge set $I$ and no isolated vertices.

Let $G$ be a (not necessarily simple) graph with vertex set $V$ and edge set $E$.
The \textit{bicircular matroid} of $G$ is the matroid $B(G)$ with base set $E$
whose independent sets are the edge sets $I\subseteq E$ such that  $G[I]$
contains at most one cycle in every connected component. 
Equivalently, the circuits of $B(G)$ are the edge sets of subgraphs which are
subdivisions of one of the graphs: two loops on the same vertex, two loops joined
by an edge, or three parallel edges joining a pair of vertices.

Mathews \cite{matthewsQJMOS28} noticed that there are only a few uniform
bicircular matroids. 

\begin{theorem}\label{thm:uniformbicircular}\cite{matthewsQJMOS28}
The uniform bicircular matroids are precisely the following:
\begin{itemize}
	\item $U_{1,n}$, $U_{2,n}$, $U_{n,n,}$ ($n\ge 0$);
	\item $U_{n-1,n}$ ($n\ge 1$);
	\item $U_{3,5}$, $U_{3,6}$ and $U_{4,6}$.
\end{itemize}
\end{theorem}

Recall that if a matroid $M$ has a double circuit of degree $k$ then 
$M$ contains a $U_{k-2,k}$ minor (Observation~\ref{obs:dcseries}).

\begin{corollary}\label{cor:6}
The  degree of a double circuit in a bicircular matroid is at most $6$.
\end{corollary}

Suppose that a graph $G$ is obtained by subdividing edges of a 
graph $H_G$ with minimum degree $3$. It is not hard to notice that
$H_G$ is unique up to isomorphism. The \textit{subdivision classes} of
$G$ are the sets of edges that correspond to a series of subdivisions of
an edge in $H_G$. An \textit{unsubdivided edge} of $G$ is an edge of $G$
that is an edge of $H_G$. Clearly, if $G$ has no
leaves an edge $xy$ is an unsubdivided edge of $G$
if and only if  $d_G(x),d_G(y)\ge 3$.

\begin{lemma}\label{lem:basic}
Let $G$ be a graph and $D\subseteq E$ a double circuit of $B(G)$.
Then $G[D]$ has no leaves and contains at most $4$ vertices
$x_1, x_2,x _3$ and $x_4$ of degree greater than or equal to $3$. Moreover,
every subdivision class of $G[D]$ belongs to the same class of the circuit partition
of $D$.
\end{lemma}
\begin{proof}
Since every element of $D$ belongs to a circuit of $D$, then $D$ has no
coloops so $G[D]$ has no leaves.
Let $V'$ be the vertex set of $G[D]$ and $r'$ the rank of $D$, so $r'= |V'|$.
Since $D$ has no leaves then $d(x)\ge 2$ for every $e\in V'$. Let $t$ be
the number of vertices in $V'$ with degree at least $3$ in $G[D]$. 
By the handshaking lemma, $2|D| \ge 3t + 2(r'-t)$. Since  $D$ is a double circuit
then $r' = |D| -2$, and thus  $2(r'+2) \ge 3t + 2(r'-t)$, so $t \le 4$. 
Finally, let $S$ be a subdivision class of $G[D]$. Then $S$ is the edge set of a
path $P$ path such that the internal vertices of $P$ have degree $2$ in $G[D]$.
Thus, every cycle in $G[D]$ that contains an edge of $P$ contains all edges of $P$. 
Hence, every circuit of $B(G)$ in $D$ that contains an edge in $P$ contains
all of them, so $E(P)$ is contained  in some circuit class of $D$.
\end{proof}

Given a double circuit $D$ of a bicircular matroid $B(G)$, a
\textit{distinguished vertex} of $D$ is a vertex of degree at least $3$ in $G[D]$.
Since $G[D]$ has no leaves, the subdivision classes of $G[D]$ correspond to
paths that contain distinguished vertices (only) as endpoints. In particular,
unsubdivided edges of $G[D]$ are edges incident in distinguished vertices.

\begin{theorem}\label{thm:main}
Let $G$ be a graph. If $\girth(G) \ge 5$ then $B(G)$ has no positive double circuits.
\end{theorem}
\begin{proof}
We proceed by contrapositive. Suppose that $D$ is a positive double circuit in $B(G)$
of degree $k$ with circuit partition $(D_1,\dots D_k)$ where $k\le 6$ by
Corollary~\ref{cor:6}.
If $k \in \{1,2,3,4\}$ with out loss of generality assume that $(D_1,\dots, D_{k-1})$
are singular circuit classes of $D$. Since the union $D'$ of these classes is $D\setminus D_k$, 
then this union is a circuit of $D$. Thus, $G[D']$ contains two cycles of $G$, so
 $\girth(G)\le |D'| \le 3$. 

Now suppose that $k \in\{5,6\}$ and $(D_1,\dots, D_k)$ has at least $4$ simple classes
$\{e_1\},\dots \{e_4\}$.  By the moreover  statement of Lemma~\ref{lem:basic}, the
edges $e_1$, $e_2$, $e_3$ and $e_4$ must be unsubdivided edges of $G[D]$. 
Thus, the endpoints of these edges are distinguished vertices of $G[D]$, which
by the same lemma there are at most $4$ of these vertices. Putting all of this 
together we conclude that $D'$ is a set of four edges such that $G[D']$ has
at most $4$ vertices. Therefore, $G[D']$ contains at least one cycle of
$G$, and so $\girth(G)\le |D'| \le 4$.

The only remaining case is when the degree of $D$ is $5$ and it has $3$ singular
classes. In this case there are $3$ unsubdivided edges $e_1,e_2,$ and $e_3$
of $G[D]$. We claim that  $\{e_1,e_2,e_3\}$ contains a cycle of $G$, and thus
$\girth(G)\le 3$. Anticipating a contradiction, suppose that $\{e_1,e_2,e_3\}$ does
not contain a cycle of $G$. This implies that $D$ has four distinguished
$x_1$, $x_2$, $x_3$ and $x_4$. Notice that if we contract a subdivided edge
$e$ of $G[D]$ we obtain
a double circuit $D'$ of $B(G)/e$ with the same degree as $D$. Inductively, we end
up with a double circuit $D_0$ and a graph $H$ with edge set $D_0$ such that 
$\{e_1,e_2,e_3\}\subseteq  D_0$ and $V(H) = \{x_1,x_2,x_3,x_4\}$. Moreover, 
$\{e_1,e_2,e_3\}$ spans a tree of $H$. On the other hand,
each edge $e_i$ for $i\in\{1,2,3\}$ belongs to a singular class of the circuit partition
of $D_0$. Since the rank of $D_0$ is $4$ then
$D_0$ contains at most six edges. Also, the circuit partition of $D_0$ has $5$ classes, 
so there must be an edge  $e_4\in D_0\setminus\{e_1,e_2,e_3\}$ that belongs to 
a singular class. Notice that that $\{e_1,e_2,e_3,e_4\}$ contains at most one
cycle of $H$ since $\{e_1,e_2,e_3\}$ spans a tree of $H$. On the other hand,
$D_0\setminus \{e_1,e_2,e_3,e_4\}$ is a class of $D_0$, so
$\{e_1,e_2,e_3,e_4\}$ is a circuit of $D_0$, i.e.\ $\{e_1,e_2,e_3,e_4\}$ contains
two cycles of $H$. We arrive at this contradiction by assuming that
$\{e_1,e_2,e_3\}$ does not contain a cycle of $G$,  thus $\girth(G)\le 3$, and
the theorem follows.
\end{proof}

This statement yields a large class $\mathcal{C}$ of bicircular matroids with no positive
double circuits.  For instance, if $G$ is the dodecahedron graph and $P$ the
Petersen graph (Figure~\ref{fig:petdod}) then $B(G)$
and $B(P)$ are cosimple bicircular matroids of corank at least two  that do not
have positive double circuits. Dually, $B(G)^\ast$
and $B(P)^\ast$ are simple matroids of rank at least two  that do not
have positive colines.


\begin{figure}[ht!]
\begin{center}

\begin{tikzpicture}[scale=0.6]

\begin{scope}[xshift=-6cm, yshift = 0cm, scale=0.7]
\node [vertex] (0) at (0:3){};
\node [vertex] (1) at (288:3){};
\node [vertex] (2) at (216:3){};
\node [vertex] (3) at (144:3){};
\node [vertex] (4) at (72:3){};

\node [vertex] (10) at (0:5){};
\node [vertex] (11) at (288:5){};
\node [vertex] (12) at (216:5){};
\node [vertex] (13) at (144:5){};
\node [vertex] (14) at (72:5){};

\foreach \from/\to in {0/2, 0/3, 1/3, 1/4, 2/4}
\draw [edge] (\from) to (\to);
\foreach \from/\to in {10/11, 12/13, 11/12, 13/14, 10/14}
\draw [edge] (\from) to (\to);
\foreach \from/\to in {0/10, 1/11, 2/12, 14/4, 3/13}
\draw [edge] (\from) to (\to);

\end{scope}

\begin{scope}[xshift=6cm, yshift=0cm, scale=0.6]
\node [vertex] (0) at (0:2){};
\node [vertex] (1) at (288:2){};
\node [vertex] (2) at (216:2){};
\node [vertex] (3) at (144:2){};
\node [vertex] (4) at (72:2){};

\node [vertex] (01) at (0:4){};
\node [vertex] (101) at (36:5){};
\node [vertex] (41) at (72:4){};
\node [vertex] (141) at (108:5){};
\node [vertex] (31) at (144:4){};
\node [vertex] (131) at (180:5){};
\node [vertex] (21) at (216:4){};
\node [vertex] (121) at (252:5){};
\node [vertex] (11a) at (288:4){};
\node [vertex] (111) at (324:5){};

\node [vertex] (10) at (36:7){};
\node [vertex] (11) at (324:7){};
\node [vertex] (12) at (252:7){};
\node [vertex] (13) at (180:7){};
\node [vertex] (14) at (108:7){};

\foreach \from/\to in {0/1, 2/3, 1/2, 3/4, 0/4}
\draw [edge] (\from) to (\to);
\foreach \from/\to in {10/11, 12/13, 11/12, 13/14, 10/14}
\draw [edge] (\from) to (\to);
\foreach \from/\to in {101/10, 111/11, 121/12, 141/14, 131/13}
\draw [edge] (\from) to (\to);
\foreach \from/\to in {01/0, 11a/1, 21/2, 41/4, 31/3}
\draw [edge] (\from) to (\to);

\foreach \from/\to in {01/101, 101/41, 41/141, 141/31, 31/131, 131/21, 21/121, 121/11a,
				11a/111, 111/01}
\draw [edge] (\from) to (\to);

\end{scope}

\end{tikzpicture}

\caption{The Petersen graph and the dodecahedron.}
\label{fig:petdod}
\end{center}
\end{figure}
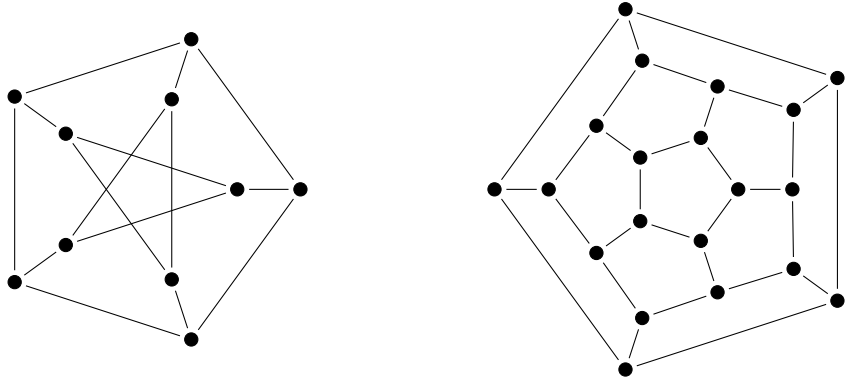

Recall that every transversal matroid is a gammoid, and since bicircular matroids
are transversal matroids~\cite{matthewsQJMOS28}, Theorem~\ref{thm:main} 
yields a class of cosimple gammoids of corank at least two  that do not have
positive double circuits.

\begin{corollary}\label{cor:counter}
Not every cosimple  gammoid of corank at least two has a positive double
circuit. Equivalently, not every simple gammoid of rank at least two has a
positive coline.
\end{corollary}


\section{Subclasses of $GSP$ oriented matroids}
\label{sec:GSP}

In this section, we show that if a matroid $M$ has a double circuit
with two singular classes, then every orientation of $M$ has a
$\{0,1,-1\}$-flow with at most two non zero entries. In particular, 
we will use this observation to show that oriented cobicircular matroids
are $GSP$.

The flow lattice or oriented uniform matroids is characterized in
\cite{hochstattlerCDM2}. The properties of these lattices that are
interesting for this work are stated in Lemma~\ref{lem:uniformflows}.
The following observation is an implication of the first part of
Lemma~\ref{lem:uniformflows}. 

\begin{observation}\label{obs:even}
Let $O$ be an oriented matroid with underlying matroid $M$. If $M$ is a series
extension of a uniform matroid of even rank, then for any series class
$C\subseteq E(M)$ there is a $\{0,1,-1\}$-flow $F$ of $O$, such that
$|F(e)| = 1$ if and only if $e\in C$. 
\end{observation}
\begin{proof}
The claim follows from the fact that for each pair $e$ and $f$ of coparallel
elements of $M$, the sign of $e$ and $f$ are either the same in every circuit
of $O$, or  are opposite in every circuit of $O$. This, 
part $1$ of Lemma~\ref{lem:uniformflows},  proves the claim.  
\end{proof}

With similar arguments to the ones in the previous proof,
and using part 2 of  Lemma~\ref{lem:uniformflows},
we conclude that the following statement holds.

\begin{observation}\label{obs:odd}
Let $O$ be an oriented matroid with underlying matroid $M$. If $M$ is a series
extension of a uniform matroid of odd rank, then for any pair $C_1$ and $C_2$ of
series classes of $M$,  there is a $\{0,1,-1\}$-flow $F$ of $O$, such that
$|F(e)| = 1$ if and only if $e\in C_1\cup C_2$. 
\end{observation}

In Observation~\ref{obs:dcseries} we argued that if $D$ is a double circuit in a
matroid $M$,  then $M[D]$ is a series extension of some uniform matroid. 
This statement together with Observations~\ref{obs:even} and~\ref{obs:odd}
imply the following statement. 

\begin{lemma}\label{lem:doublecircuitS}
Let $O$ be an oriented matroid with underlying matroid $M$.  If $D$ is
a double circuit of $M$ such that:
\begin{itemize}
	\item $D$ has even degree and at least one singular class, or
	\item $D$ has odd degree and at least two singular classes,
\end{itemize}
then $O$ has a $\{0,1,-1\}$-flow with at most two non-zero entries.
\end{lemma}

As previously mentioned, Goddyn, Hochst\"attler and Neudauer, 
show that finding positive double circuits in certain matroids
imply that the orientations of these matroids are $coGSP$ \cite{goddynDM339}. 
Lemma~\ref{lem:doublecircuitS} yields a simple tool to improve the 
previous sufficient condition as follows.

\begin{proposition}\label{prop:sufCOGSP}
Let $\mathcal{C}$ be a minor closed class of orientable matroids. 
If every cosimple matroid in $\mathcal{C}$ has a double circuit $D$ such that:
\begin{itemize}
	\item $D$ has even degree and at least one singular class, or
	\item $D$ has odd degree and at least two singular classes,
\end{itemize}
then every orientation of a matroid in $\mathcal{C}$ is $coGSP$.
\end{proposition}

For the sake of completeness, we state the dual version of
Proposition~\ref{prop:sufCOGSP}.

\begingroup
\def\thetheorem{(Dual version of Proposition~\ref{prop:sufCOGSP})}
\begin{proposition}
Let $\mathcal{C}$ be a minor closed class of orientable matroids.
If every simple matroid in $\mathcal{C}$ has a coline $L$ such that:
\begin{itemize}
	\item there is a simple copoint on $L$ and $L$ has even degree, or
	\item there is a pair of simple copoints on $L$ and $L$ has odd degree,
\end{itemize}
then every orientation of a matroid in $\mathcal{C}$ is $GSP$.
\end{proposition}
\addtocounter{theorem}{-1}
\endgroup

In the subsections below, we will use these propositions to show that certain
classes of oriented matroids are $GSP$ ($coGSP$). To do so, we begin by 
proving the following lemma which guarantees the existence of double
circuits with at least two singular classes. 

\begin{lemma}\label{lem:basicsuf}
For any matroid $M$ the following statements are equivalent:
\begin{itemize}
	\item there is a double circuit $D\subseteq M$ with at least two singular classes, and
	\item there is a pair $C_1$ and $C_2$ of circuits such that 
	$|C_1\triangle C_2| = 2$. 
\end{itemize}
\end{lemma}
\begin{proof}
If $D = (D_1,\dots, D_k)$ is a double circuit such that $|D_1| = |D_2| = 1$, 
then $D-D_1$ and $D-D_2$ are a pair of circuits such that
$|(D-D_1)\triangle (D-D_2)| = |D_1\cup D_2| = 2$.  
Now suppose that $C_1$ and $C_2$ satisfy the second statement. 
Let $D = C_1\cup C_2$. It is clear that $|D| = k+2 = r(D)+2$. Moreover, 
since any $e\in D$ belongs to $C_i$ for some $i\in\{1,2\}$, then $r(D-e)\ge
r(C_i-e) = k$. So, for any $e\in D$ the rank of $D-e$ does not decrease,
which shows that $D$ is a double circuit. Finally, notice that
$D\setminus C_1$ and $D\setminus C_2$ form two singular classes
of the circuit partition of $D$. The claim follows. 
\end{proof}

\begingroup
\def\thetheorem{(Dual version of Lemma~\ref{lem:basicsuf})}
\begin{lemma}
For any matroid $M$ the following statements are equivalent:
\begin{itemize}
	\item there is a coline $L\subseteq M$ with at least two singular classes, and
	\item there is a pair $H_1$ and $H_2$ of hyperplanes of $M$
	such that 	$|H_1\triangle H_2| = 2$. 
\end{itemize}
\end{lemma}
\addtocounter{theorem}{-1}
\endgroup

\subsection{Bicircular matroids}

In this subsection, we show that oriented bicircular matroids are $coGSP$.
To do so, we will show that for every bicircular matroid $B(G)$ there is a pair
of circuits $C_1$ and $C_2$ such that $|C_1\triangle C_2| = 2$.
We begin by proving the following lemma.

\begin{lemma}\label{lem:circuitsbic}
Consider a graph $G$. If $B(G)$ is a cosimple matroid, 
then $B(G)$ contains a pair of circuits $C_1$ and $C_2$
such that  $|C_1\triangle C_2| = 2$.
\end{lemma}
\begin{proof}
Since $G$ is cosimple, every vertex of $G$ is incident with at least $3$ edges. 
Consider a maximum path $P$ of $G$, where
$P = v_1 e_1 v_2\dots v_{k-1} e_{k-1}v_k$. Let $e_k\neq e_{k-1}$ be an edge
incident in $v_k$. By the choice of $P$, $e_k$ has both endpoints
in $P$.  Denote by $E'$ the set of edges $\{e_1,\dots, e_k\}$. 
Let $e$ and $f$ be a pair of edges incident in $v_1$ different from each 
other and different to $e_1$. By the choice of $P$, both endpoints of
$e$ and both endpoints of $f$ belong to $P$. It is straight forward to notice
that $E'\cup\{e\}$ and $E'\cup \{f\}$ are both circuits of $B(G)$. 
These circuits also satisfy that $|(E'\cup\{e\})\triangle (E'\cup \{f\})| = 2$.
Therefore, $B(G)$ contains a pair of circuits $C_1$ and $C_2$
such that  $|C_1\triangle C_2| = 2$.
\end{proof}

Lemmas~\ref{lem:basicsuf} and~\ref{lem:circuitsbic} imply that every cosimple
bicircular matroid has a  double circuit with at least two singular classes.
So, using Proposition~\ref{prop:sufCOGSP} we conclude that any oriented
bicircular matroid is $coGSP$.

\begin{proposition}\label{prop:bicircularcoGSP}
Every oriented bicircular matroid is $coGSP$.
\end{proposition}

An equivalent reformulation of Proposition~\ref{prop:bicircularcoGSP},
states that every oriented cobicircular matroid is $GSP$.
Since every $GSP$ matroid has a NZ-$3$ coflow, we conclude the
following statement.

\begin{corollary}
Every oriented cobicircular matroid has a  NZ $3$-coflow.
\end{corollary}

\subsection{Clone reducible matroids}

We conclude this section by defining the class of \textit{clone reducible matroids}.
In particular, the graphic matroids in this class correspond to graphic matroids
of series parallel graphs, also any orientation of a clone reducible matroid 
is a $GSP$  oriented matroid. 

Consider a matroid $M$ and a pair of elements $e$ and $f$ of $M$. 
We say that $e$ and $f$ are \textit{clones} (in $M$) if permuting $e$ and $f$ is
an automorphism of $M$. We say that a matroid $M$ is \textit{clone reducible}
if every minor of $M$ (with at least two elements) has a pair of clones.
For instance,  uniform matroids, matroids of rank at most $2$ and matroids of
corank at most $2$ are examples of clone reducible matroids.

It is not hard to notice that if $e$ and $f$ are clones in $M$, and $g\in E(M)-\{e,f\}$,
then $e$ and $f$ are clones in $M-g$. Also notice that if $e$ and $f$ are clones in 
$M$, then $e$ and $f$ are clones in $M^\ast$. Putting these two observations
together, we conclude that if $e$ and $f$ are clones of $M$, and $N$ is
a  minor of $M$ such that $e,f\in E(N)$, then $e$ and $f$ are clones
in $N$. This implies that $M$ is a clone reducible matroid, if and only if
there is a linear ordering $e_1 \le e_2 \le \cdots \le e_n$ of $E(M)$,
such that $e_i$ has a clone in the restriction $M[e_1,\dots, e_i]$, for
all $i\in\{2,\dots, n\}$.

Notice that if $e$ and $f$ are a pair of parallel edges in a graph $G$,
then $e$ and $f$ are clones in $M(G)$. Similarly, if $G$ has a vertex $v$
of degree two, and $e$ and $f$ are incident with $v$, then
$e$ and $f$ are also clones in $M(G)$. So, graphic matroids
of series parallel graphs ($M(K_4)$-free graphic matroids)
are clone reducible matroids.

\begin{proposition}\label{prop:Tbinary}
For a binary matroid $M$ the following statements are equivalent:
\begin{enumerate}
	\item $M$ does not contain an $M(K_4)$-minor, and
	\item $M$ is a clone reducible matroid.
\end{enumerate}
\end{proposition}
\begin{proof}
On the one hand, $M(K_4)$-free binary matroids correspond
to graphic matroids of series parallel graphs. Thus, by the arguments
preceding this statement, the first item implies the second one. 
On the other hand, $M(K_4)$ does not contain a pair of clones, thus
the second statement implies the first one.
\end{proof}

We have previously observed that the class of clone reducible
matroids is closed under duality. Now we show that every 
orientation of a clone reducible matroid is $GSP$, and thus
$coGSP$.

\begin{observation}\label{obs:T}
Every orientation of a clone reducible matroid is a $GSP$ and 
$coGSP$ oriented matroid.
\end{observation}
\begin{proof}
It suffices to show that any orientation of a clone reducible matroid
is $coGSP$. Suppose that $e$ and $f$ are 
pair of clones in a cosimple clone reducible matroid $M$.
In particular, 
$e$ and $f$ are not coparallel elements, so there is a circuit
$C_1$ such that $e\in C_1$ but $f\not\in C_1$. Since $e$ and $f$
are clones, then the set $C_2$ defined by $(C_1-e)\cup\{f\}$ is a circuit
of $M$. Thus, $C_1$ and $C_2$ are a pair of circuits of $M$ such that
$|C_1\triangle C_2| = 2$. Thus, by Lemma~\ref{lem:basicsuf}, 
$M$ has a double circuit with two singular classes. So, by
Proposition~\ref{prop:sufCOGSP} we conclude that every orientation
of a clone reducible matroid is a $coGSP$ oriented matroid. 
\end{proof}

In \cite{albrechtPhd} the author shows that orientations of lattice path
matroids  are $GSP$. We propose a simple proof of this fact by
showing that every lattice path matroid is a clone reducible matroid.

\begin{lemma}\label{lem:clones}
Every lattice path matroid on at least two elements has a pair of clones. 
\end{lemma}
\begin{proof}
Consider a lattice path matroid $L$ on $[n]$, and suppose that $L$
has no coloops. It is not hard to notice that if $I$ is an independent
set such that $1\in I$, then $(I-1)\cup\{2\}$ is also an independent set. 
With out loss of generality suppose that $I$ is a base. If $2\in I$, then
$(I-1)\cup\{2\} = I-1$ and the claim is trivial.  On the other hand,
if $2\not\in I$, by the maximality of $I$,  there is a lattice path $P$ whose
north steps include $1$ and its east steps include $2$, i.e., 
$P = n_1e_2Q$, for some lattice path $Q$ starting in $(1,1)$.
Thus, the lattice path $e_1n_2Q$ is a lattice path which certifies
that $(I-1)\cup\{2\}$ is an independent set in $I$. With similar arguments
we observe that if $I$ is an independent set such that $2\in I$, then
$(I-2)\cup\{1\}$ is also an independent set. Putting these observations
together, we conclude that if $C$ is a circuit such that $1\in C$
and $2\not\in C$, then $C-1\cup\{2\}$ is a circuit (and vicecersa). 
Therefore, the function $f\colon [n]\to [n]$ that fixes $[n]\setminus\{1,2\}$
and permutes $1$ with $2$, maps circuits $M$ to circuits of $M$,
i.e., $f$ is an automorphism of $M$. 
Thus, $1$ and $2$ are clones in $M$. 
\end{proof}

\begin{proposition}
Every orientation of a lattice path matroid is a $GSP$ oriented matroid.
\end{proposition}

\begin{corollary}
Every orientation of a lattice path matroid has a $NZ$ $3$-coflow.
\end{corollary}

\subsection{Rank $3$ matroids}

Consider a rank $3$ simple and cosimple non-uniform oriented matroid $O$ on the
set $E$.  Hochst\"attler and Nickel~\cite{hochstattlerCDM2}, show that
if $O$ is not an orientation of $M(K_4)$,  then the flow lattice
$\mathcal{F}_O$ is $\mathbb{Z}^E$.
In particular, this implies that in this case, $O$ has a $\{0,1,-1\}$-flow
with at most one non-zero entry. 

\begin{proposition}
Every $M(K_4)$-free oriented matroid of rank at most $3$ is $coGSP$.
Dually, every $M(K_4)$-free oriented matroid of corank at most $3$ is $GSP$
\end{proposition}
\begin{proof}
Recall that we want to show that every such cosimple oriented matroid $O$
has a $\{0,1,-1\}$-flow with at most two non-zero entries. 
The arguments above this proposition takes care of the case when $O$ is a simple
rank $3$ non-uniform oriented matroid. Now, if $O$ has a circuit
$C$ of size at most two, then $C$ is a $\{0,1,-1\}$-flow with at most two non-zero
entries. Finally, if $O$ is uniform or has rank
at most $2$, then every minor of the underlying matroid of $O$ has a pair
of clones.  So, the claim follows by Observation~\ref{obs:T}.
\end{proof}


%
%
%
%
%
%
%
%
%
%
%
%
%


\section{Conclusions}
\label{sec:conclusions}

A simple observation used to prove Theorem~\ref{thm:main} shows that the degree
of double circuits in bicircular matroids is bounded above by $6$. This raises
the natural problem of describing the classes of matroids obtained by considering
bicircular matroids whose double circuits have degree at most $k$ where
$k\in\{3,4,5\}$. This question has been answered from
another perspective for $k = 3$: the positive double circuits in a bicircular matroid
$B$ have degree at most $3$ if and only if $B$ is a binary bicircular matroid
\cite{matthewsQJMOS28}.
If we also restrict the degree of colines, for the case $k = 4$ we recover
ternary bicircular matroids \cite{sivaramanDM328}. For the case $k = 5$, considering
bicircular matroids with double circuits and colines  of degree at most $k$, we do
not recover representation over $GF(4)$ since $P_6$ has no positive double
circuits nor colines of degree greater that or equal to $5$, but is a bicircular matroid
not representable over $GF(4)$ \cite{chunDM339}. 
We are interested in knowing if there is a meaningful description of 
bicircular matroids whose double circuits have bounded degree.


\begin{problem}
Provide a meaningful description of bicircular matroids where every double
circuit has degree at most $k$ for  $k\in\{4,5\}$.
\end{problem}

The motivation of this work was to settle Conjecture~\ref{conj:coline};
we showed that it does not hold even for cobicircular matroids. 
Nonetheless, we prove that all oriented cobicircular matroids
are $GSP$. We did so by using Proposition~\ref{prop:sufCOGSP},
which yields a weaker condition (to that in Conjecture~\ref{conj:coline})
for proving that certain classes oriented matroids are $GSP$. We propose
the following question.

\begin{question}
Is it true that every cosimple gammoid $M$ contains a pair of circuits $C_1$ and
$C_2$ such that $|C_1\triangle C_2| = 2$?
\end{question}

At the end of Section~\ref{sec:GSP}, we introduce the class of 
clone reducible matroids.
In particular, we showed that lattice path matroids and graphic matroids of
series parallel are examples of clone reducible matroids.
Since lattice path matroid and graphic matroids
form a pair of incomparable matroid classes, then the class of
clone reducible matroids properly contains both classes. We also noticed that
$M(K_4)$ does not have a pair of clones, and it is straightforward to observe
that the three whirl does not have a pair of clones either.
What are the excluded minors to the class of clone reducible matroids?

\begin{problem}
Characterize the class of clone reducible matroids.
\end{problem}

\section*{Acknowledgements}
The authors gratefully acknowledge discussions with Luis Goddyn. In particular, 
he proposed Lemma~\ref{lem:circuitsbic} with a different proof method.
This work was carried out during a visit of the first author
at FernUniversit\"at in Hagen, supported by DAAD grant 57552339.


\begin{thebibliography}{15}

\section*{References}

\bibitem{albrechtPhd}
	I.~Albrecht,
	Contributions to the Problems of Recognizing and Coloring Gammoids,
	Doctoral Dissertation, FernUniversit\"at in Hagen, 2018.
	https://doi.org/10.18445/20180820-090543-4


\bibitem{bjorner1993}
	 A.~Bj\"orner, M.~L.~Vergnas, B.~Sturmfels, N.~White, G.~M.~Ziegler,
	 Oriented Matroids,
	 Cambridge University Press, 1993.

\bibitem{bondy2008}
	J.A.~Bondy and U.S.R~Murty,
	Graph Theory,
	Springer, Berlin, 2008.
	
	
\bibitem{chunDM339}
	D.~Chun, T.~Moss, D.~Slilaty, X.~Zhou,
	Bicircular Matroids representable over $GF(4)$ and $GF(5)$, 
	Discrete Mathematics 339(9) (2016) 2239--2248.
	
\bibitem{dressC7}
	A.~Dress and L.~Lov\'asz,
	On some combinatorial properties of algebraic matroids, 
	Combinatorica 7(1) (1987) 39--48.
	
\bibitem{goddynSM}
	L.~Goddyn and W.~Hochst\"attler,
	Nowhere-zero flows in regular matroids and Hadwiger's conjecture,
	Seminarberichte der Mathematik -- FernUniversit\"at in Hagen 87 (2015)
	97--102.
	
	
\bibitem{goddynDM339}
	L.~Goddyn, W.~Hochst\"attler, N.~Neudauer,
	Bicircular matroids are $3$-colourable,
	Discrete Mathematics 339(5) (2016) 1425--1429.
	
\bibitem{hochstattlerEJC27}
	W.~Hochst\"attler and J.~Ne\v{s}et\v{r}il,
	Antisymmetric flows in matroids,
	European Journal of Combinatorics 27(7) (2006) 1129--1134.
	
	
\bibitem{hochstattlerCDM2}
	W.~Hochst\"attler and  R.~Nickel,
	The flow lattice of oriented matroids,
	Contributions to Discrete Mathematics 2(1) (2007) 68--86.

	
\bibitem{ingletonJCTB15}
	A.W.~Ingleton, 
	Gammoids and Transversal Matroids, 
	Journal of Combinatorial Theory (B) 15 (1973) 51--68.
	
\bibitem{matthewsQJMOS28}
	L.~R.~Matthews,
	Bicircular Matroids,
	Quarterly Journal of Mathematics (2) 28(110) (1997) 213--227.
	
\bibitem{oxley1992}
	J.G.~Oxley, 
	Matroid Theory, 
	The Clarendon Press Oxford University Press, New York (1992).
	
\bibitem{robsertsonC13}
	N.~Robertson, P.~Seymour, R.~Thomas,
	Hadwiger's conjecture for $K_6$-free graphs,
	Combinatorica 13 (1993) 279--361.
	
\bibitem{sivaramanDM328}
	V.~Sivaraman,
	Bicircular signed-graphic matroids,
	Discrete Mathematics 328 (2014) 1--4.
	
\bibitem{tutteCJM6}
	W.~Tutte,
	A contribution to the theory of chromatic polynomials,
	Canadian Journal of Mathematics 6 (1954)  80--91.
	
\bibitem{tutteCMA}
	W.~Tutte,
	A geometrical version of the four color problem,
	R.C. Bose, T.A.\ Dowling (Eds.), Combinatorial Mathematics and its
	Applications, University of North Carolina Press, Chapel Hill, NC (1967)
	553--560.

\bibitem{tutteJCT}
	W.~Tutte,
	On the algebraic theory of graph colorings,
	Journal of Combinatorial Theory 1(1) (1966) 15--50.

	
   
\end{thebibliography}
\end{document}